\documentclass[a4paper,12pt]{amsart}

\usepackage{amsfonts}
\usepackage{amssymb}
\usepackage{amsmath}
\usepackage{amsthm}

\DeclareFontFamily{OT1}{pzc}{}
\DeclareFontShape{OT1}{pzc}{m}{it}{<-> s * [1.10] pzcmi7t}{}
\DeclareMathAlphabet{\mathpzc}{OT1}{pzc}{m}{it}

\newtheorem{theorem}{Theorem}[section]
\newtheorem{proposition}{Proposition}[section]
\newtheorem{lemma}{Lemma}[section]

\newtheorem{remark}{Remark}[section]

\numberwithin{equation}{section}
\numberwithin{theorem}{section}
\numberwithin{lemma}{section}
\numberwithin{corollary}{section}
\numberwithin{proposition}{section}
\numberwithin{remark}{section}

\newcommand{\x}{\mathpzc{x}}
\newcommand{\f}{\mbox{\small{$\mathpzc{F}$}}}
\newcommand{\w}{\mbox{\scriptsize{$\mathpzc{W}$}}}
\newcommand{\sT}{\mathcal T}

\newcommand{\ca}{\mathpzc{a}}
\newcommand{\cb}{\mathpzc{b}}
\newcommand{\cc}{\mathpzc{c}}
\newcommand{\cB}{\mathcal B}
\newcommand{\wC}{\widetilde{C}}
\newcommand{\wD}{\widetilde{D}}
\newcommand{\wcD}{\widetilde{\mathcal D}}
\newcommand{\wM}{\widetilde{M}}
\newcommand{\wV}{\widetilde{V}}
\newcommand{\wP}{\widetilde{P}}
\newcommand{\wcP}{\widetilde{\mathcal P}}
\newcommand{\cP}{\mathcal P}

\begin{document}

\title[Higher Order Approximation by a Modified MKZ Operator]
      {Higher Order Approximation of Continuous Functions \\ by a Modified Meyer-K\"{o}nig and Zeller-Type Operator}

\author[I. Gadjev]{Ivan Gadjev$^1$}
\address{$^1$\,Faculty of Mathematics and Informatics, Sofia University St. Kliment Ohridski, \newline
         \indent 5, J. Bourchier Blvd, 1164 Sofia, Bulgaria}
\email{gadjev@fmi.uni-sofia.bg}

\author[P. Parvanov]{Parvan Parvanov$^2$}
\address{$^2$\,Faculty of Mathematics and Informatics, Sofia University St. Kliment Ohridski, \newline
         \indent 5, J. Bourchier Blvd, 1164 Sofia, Bulgaria}
\email{pparvan@fmi.uni-sofia.bg}

\author[R. Uluchev]{Rumen Uluchev$^3$}
\address{$^3$\,Faculty of Mathematics and Informatics, Sofia University St. Kliment Ohridski, \newline
         \indent 5, J. Bourchier Blvd, 1164 Sofia, Bulgaria}
\email{rumenu@fmi.uni-sofia.bg}

\thanks{$^{1,2,3}$\,This study is financed by the European Union-NextGenerationEU, through the
        National Recovery and Resilience Plan of the Republic of Bulgaria, project No
        BG-RRP-2.004-0008.}



\date{}

\begin{abstract}
  A new Goodman-Sharma type modification of the Meyer-K\"{o}nig and Zeller operator for
  approximation of bounded continuous functions on $[0,\,1)$ is presented. We estimate the
  approximation error of the proposed operator and prove direct and strong converse theorems with
  respect to a related K-functional. The operator is linear but not a positive one. However it
  benefits a better order of approximation compared to the Goodman-Sharma variant of
  Meyer-K\"{o}nig and Zeller type operator investigated by Ivanov and Parvanov \cite{IvPa2012}.
\end{abstract}

\maketitle

%====================================================================================================

\section{Introduction} \label{s1}

The Meyer-Konig and Zeller operator (MKZ operator) was introduced in 1960, see \cite{MeKoZe1960}.
A minor modification by Cheney and Sharma \cite{ChSh1964} has led to the following definition of the
MKZ operator for functions $f\in C[0,1)$ and $n\in\mathbb{N}$:
\begin{equation} \label{eq:1.1}
  \begin{gathered}
    M_n^{[MKZ]}(f,x) = \sum_{k=0}^{\infty} f\Big(\frac{k}{k+n}\Big) P_{n,k}(x), \qquad x\in[0,1), \\
    P_{n,k}(x) = \binom{n+k}{k} x^k(1-x)^{n+1}.
  \end{gathered}
\end{equation}
Subsequently the MKZ operator was an object of intensive investigations and we refer the
reader, e.g. to the paper of Totik \cite{To1983} where direct and converse results were proved.

Motivated by the integration of Bernstein polynomials by Derriennic \cite{De1981}, some authors
have modified the MKZ operator to approximate Lebesgue integrable functions on $(0,1)$. Later
Heilmann \cite{He2003b} and  Abel, Gupta and Ivan \cite{AbGuIv2004} merged these Durrmeyer-type
variants into the family of operators defined for $\alpha \in \mathbb{R}$ and $\beta \in \mathbb{Z}$
by
\begin{align} \label{eq:1.2}
  \wM_{n,\alpha,\beta} (f,x) =
  (n+\alpha)\sum_{k=\max\{0,-\beta\}}^{\infty} \bigg(\int_0^1 P_{n+\alpha,k+\beta}(t)(1-t)^{-2}f(t)\,dt\bigg)P_{n,k}(x),
\end{align}
for all $n\in \mathbb{N}$ with $n+\alpha>0$. Let us mention a few special cases:

\smallskip
\begin{itemize}
  \item $\wM_{n,2,0} \equiv \wM^{\;\text{[Chen]}}_n$, introduced by Chen \cite{Ch1986};

  \smallskip
  \item $\wM_{n-1,0,-2} \equiv \wM^{\;\text{[Guo]}}_n$, proposed by Guo \cite{Gu1988} and studied in \cite{GuAh1995,Gu1995};

  \smallskip
  \item $\wM_{n-1,1,0} \equiv \wM^{\;\text{[G-A]}}_n$, introduced by Gupta and Abel \cite{GuAb2004};

  \smallskip
  \item $\wM_{n,\alpha,0} \equiv \wM^{\;\text{[Heilmann]}}_{n,\alpha}$, suggested by Heilmann \cite{He2003a};

  \smallskip
  \item $\wM_{n,0,0} \equiv \wM^{\;\text{[Heilmann]}}_{n,0} \equiv \wM_n$, the ``natural'' MKZ-Durrmeyer operator.
\end{itemize}

\smallskip
In 2012, Ivanov and Parvanov \cite{IvPa2012} investigated the uniform weighted approximation error
of the Goodman-Sharma-type modification of MKZ operator (MKZ-GS operator) defined for
$n\in \mathbb{N}$ by
\begin{equation} \label{eq:1.3}
  \begin{gathered}
    M_n (f,x) = \sum _{k=0}^{\infty}\,u_{n,k}(f)P_{n,k}(x), \qquad x\in [0,1), \\
    u_{n,0}(f) = f(0), \qquad u_{n,k}(f) = n \int_0^1 P_{n,k-1}(t)(1-t)^{-2} f(t)\,dt,
  \end{gathered}
\end{equation}
where $f(x)$ is a Lebesgue integrable function on $(0,1)$ with a finite limit $f(0)$ as $x\to 0$.
According to \eqref{eq:1.2} we have
$$
  M_n (f,x) = f(0) P_{n,0}(x) + \wM_{n,0,-1} (f,x).
$$

The definition of the MKZ-GS operator follows an idea of Chen~\cite{Ch1987} and Goodman and Sharma
\cite{GoSh1988,GoSh1991} (see also \cite{PaPo1994}) for a modification of the Bernstein-Durrmeyer
operator \cite{Du1967} which recovers the linear functions. Similarly, the operator considered in \cite{Fi2005}
by Finta (see also \cite{IvPa2011}) is a Goodman-Sharma-type variant of the Baskakov-Durrmeyer
operator.

Let us denote, as usual, by
$$
  \varphi(x) = x(1-x)^2
$$
the weight function which is naturally associated with the second order derivative of the MKZ-type
operators. We set $D f(x)=f'(x)$, $D^2 f(x)=f''(x)$,
\begin{equation} \label{eq:1.4}
  \wD f(x) := \varphi(x)f''(x)
\end{equation}
and also
$$
  \wD^2 f(x) := \wD \big(\wD f(x)\big), \qquad \wD^3 f(x) := \wD \big(\wD^2 f(x)\big).
$$

The MKZ-GS operator combines nice properties both of the MKZ operator and of its Durrmeyer-type
modifications. Thus, the MKZ-GS operator $M_n^{[MKZ]}$ recovers the linear functions and is suitable
for uniform approximation. Moreover, $M_n$ just like the MKZ-Durrmeyer operator commutes in the
sense $M_n M_k f = M_k M_n f$, as well as it commutes with the differential operator $\wD$, i.e.
$M_n \wD f = \wD M_n f$.

Recently Acu and Agrawal \cite{AcAg2019} and Jabbar and Hassan \cite{JaHa2024} studied
operator families of Bernstein-Durrmeyer type and of Baskakov-type, respectively, where basis
functions in the definitions are replaced by linear combinations of basis functions of lower degree
with coefficients being polynomials of appropriate degree. The advantage is obtaining a better order
of approximation than the classical operators.

The ideas presented in \cite{AcAg2019} and \cite{JaHa2024} prompted the authors of the current paper
to explore a new modification of the MKZ-GS operator, explicitly defined by
\begin{equation} \label{eq:1.5}
  \begin{gathered}
    \wM_n(f,x) = \sum_{k=0}^{\infty} u_{n,k}(f)\wP_{n,k}(x), \qquad x\in[0,1), \\
    \wP_{n,k}(x) = P_{n,k}(x) - \frac{1}{n}\,\wD P_{n,k}(x).
  \end{gathered}
\end{equation}

To state our results we introduce some notations and definitions. Let $L_{\infty}[0,1)$ be the
space of all Lebesgue measurable and essentially bounded functions on $[0,1)$ and $AC_{loc}(0,1)$
consists of the functions absolutely continuous on any subinterval
$[a,b]\subset(0,1)$. We set
$$
  W^2(\varphi)[0,1) := \big\{g\,:\,g,g'\in AC_{loc}(0,1), \ \wD g\in L_{\infty}[0,1)\big\}.
$$
By $W^2_0(\varphi)[0,1)$ we denote the subspace of $W^2(\varphi)[0,1)$ of functions $g$ satisfying
the additional boundary condition
$$
  \lim_{x\to 0^{+}} \wD g(x) = 0.
$$

From now on we will use $\|\cdot\|$ to mean the uniform norm on the interval $[0,1)$. Sometimes
we will specify the interval $[0,a)$ of the uniform norm $\|\cdot\|_{[0,a)}$ with $a=1$ or
$a=\infty$.

Also, we will use the notations
\begin{itemize}
  \item \ $x$ \ for the variable if we consider functions $f$ with domain $[0,1)$,
  \item \ calligraphic symbols: $\x$ for the variable of functions $\f,\w$ with domain $[0,\infty)$.
\end{itemize}

For functions $f\in C[0,1)$ and $t>0$ we define the K-functional
\begin{equation} \label{eq:1.6}
  K(f,t) := \inf \big\{\|f-g\|+t\|\wD^2 g\|\,:\,g\,\in W^2_0(\varphi)[0,1),\,\wD g\in W^2(\varphi)[0,1)\big\}.
\end{equation}

The goal of our study in this work is to estimate the approximation error for functions
$f\in C[0,1)$ using the modified MKZ-GS operator \eqref{eq:1.5}. In particular, by involving the
K-functional \eqref{eq:1.6}, we prove a direct inequality and a strong converse inequality of
Type~B, according to the strong converse inequalities classification due to Ditzian and Ivanov in
\cite{DiIv1993}. Our main results read as follows.

\begin{theorem} \label{th:1.1}
  If $\,n\in\mathbb{N}$, $n\ge 2$, and $f\in C[0,1)$, then
  $$
    \big\|\wM_n f-f\big\| \le \big(\sqrt{6}+1\big)\,K\Big(f,\frac{1}{n^2}\Big).
  $$
\end{theorem}

\begin{theorem} \label{th:1.2}
  For every function $f\in C[0,1)$ and $n\in\mathbb{N}$, $n\ge 17$, there exist constants $C,L>0$
  such that
  $$
    K\Big(f,\frac{1}{n^2}\Big) \le
    C\,\frac{\ell^2}{n^2}\big(\big\|\wM_n f - f \big\| + \big\|\wM_{\ell} f - f) \big\| \big).
  $$
  for all $\ell\ge Ln$.
\end{theorem}

\begin{remark}
  We can also formulate the results of Theorem~\ref{th:1.1} and Theorem~\ref{th:1.2} in the manner:
  there exists a natural number $k$ such that
  $$
    K\Big(f,\frac{1}{n^2}\Big) \sim \big\|\wM_n f - f \big\| + \big\|\wM_{kn} f - f \big\|.
  $$
\end{remark}

Briefly, the paper is organized as follows. In Section~1 we give some historical notes, notations,
introduce the modified MKZ-GS operator $\wM_n$ and claim the main results. In Section~2 we consider
a particular transformation mapping functions defined on $[0,\infty)$ into functions defined on
$[0,1)$, which will help us to transfer results for the Goodman-Sharma modification of Baskakov
operator to their analogues for the modified MKZ-GS operator and vice versa. Preliminary and
auxiliary results for MKZ and Baskakov operators are presented in Section~3 and Section~4,
respectively. Section~5 includes an estimation of the norm of the operator $\wM_n$, a Jackson type
inequality and a proof of the direct inequality stated in Theorem~\ref{th:1.1}. The last Section~6
is devoted to a converse result for the modified MKZ-GS operator \eqref{eq:1.5}. Inequalities of
the Voronovskaya type and Bernstein type concerning the approximation operator $\wM_n$ and the
differential operator~$\wD$, defined in \eqref{eq:1.4}, are proved. Theorem~\ref{th:1.2} represents
a strong converse inequality. A complete proof of the converse theorem is given.

%====================================================================================================

\medskip
\section{Relations between the Goodman-Sharma modification of Baskakov operator and the MKZ operator}
\label{s2}

A well-know (see \cite{To1983}, \cite{He2003a,He2003b,He2006}) and useful mapping of functions
defined on $[0,\infty)$ into functions with domain $[0,1)$ will help us to transfer our results for
the Goodman-Sharma modification of the Baskakov operator (see \cite{GaPaUl2025}) to their analogues
for the modified MKZ-GS operator.

Henceforth, if we have several operators of any kind (approximation, differential, etc.), written
consecutively in an expression, we will use the convention that this is the composition of the
operators and usually omit the brackets.

Let us consider the change of the variable $\sigma : [0,1)\to[0,\infty)$ with
$$
  \x = \sigma(x) := \frac{x}{1-x}\,.
$$
Then the inverse change of the variable $\sigma^{-1} : [0,\infty)\to [0,1)$ is given by
$$
  x = \sigma^{-1}(\x) = \frac{\x}{1+\x}\,.
$$

An operator $\sT$ transforming a function $\f:[0,\infty)\to \mathbb{R}$ into a function
$f:[0,1)\to \mathbb{R}$ is defined by
\begin{equation} \label{eq:2.1}
   \big(\sT \f\big)(x) = (1-x)\f\Big(\frac{x}{1-x}\Big) = f(x), \qquad x\in [0,1),
\end{equation}
Clearly, for the inverse operator $\sT^{-1}$ we have
\begin{equation} \label{eq:2.2}
  \big(\sT^{-1}f\big)(\x) = (1+\x)f\Big(\frac{\x}{1+\x}\Big) = \f(\x), \qquad \x\in [0,\infty).
\end{equation}

We mentioned above that $\wD=\varphi D^2$, $\varphi (x)=x(1-x)^2$, naturally associates with
the MKZ type operators. Similarly, the corresponding differential operator for the Baskakov type
operators has the form
$$
  \wcD = \psi D^2, \qquad \psi(\x) = \x(1+\x).
$$
Important relationships for the operators $\wD$, $\wcD$, $\sT$ and $\sT^{-1}$ are the following:
\begin{gather}
  \sT \wcD\f = \sT(\psi D^2\f) = \varphi D^2 \sT\f = \wD \sT\f, \label{eq:2.3} \\
  \sT^{-1} \wD f = \sT^{-1}(\varphi D^2 f) = \psi D^2 \sT^{-1} f = \wcD \sT^{-1} f. \label{eq:2.4}
\end{gather}

From \eqref{eq:2.1}--\eqref{eq:2.4} we obtain:
\begin{proposition} \label{pr:2.1}
  If $\,\w(\x)=(1+\x)^{-1}$, $\x\in [0,\infty)$, then $\sT:W^2(\w\psi)[0,\infty)\to W^2(\varphi)[0,1)$
  is a bijective mapping and there hold
  $$
    \big\|\wD \sT \f \big\|_{[0,1)} = \big\|\w \wcD\f\big\|_{[0,\infty)}\,, \qquad
    \big\|\w\wcD \sT^{-1} f \big\|_{[0,\infty)} = \|\wD f\|_{[0,1)}\,,
  $$
  i.e.
  \begin{equation} \label{eq:2.5}
    \big\|\wD f\big\|_{[0,1)} = \big\|\w \wcD\f\big\|_{[0,\infty)}\,.
  \end{equation}
\end{proposition}

The classical Baskakov operator \cite{Ba1957} is defined by
\begin{equation} \label{eq:2.6}
  \cB_n(\f,\x) = \sum_{k=0}^{\infty} \f\Big(\frac{k}{n}\Big)\,\cP_{n,k}(\x), \qquad
  \cP_{n,k}(\x) = \binom{n+k-1}{k}\x^k(1+\x)^{-n-k},
\end{equation}
for $\x\in [0,\infty)$.

The following statement gives a relation between the MKZ operator $M_n^{[MKZ]}$ and the Baskakov
operator $\cB_n$.

\begin{proposition}[{{\cite[Proposition~7]{IvPa2012}}}] \label{pr:2.2}
  For every function $f$ such that one of the series (representing $M_n^{[MKZ]} f$ and
  $\cB_n \f$) in \eqref{eq:2.7} is convergent and for every $n\in\mathbb{N}$ we have
  \begin{equation} \label{eq:2.7}
    M_n^{[MKZ]}(f,x) % = \sT\big(\cB_n(\sT^{-1} f,x)\big)
                     = (1-x)\cB_n\Big(\f,\frac{x}{1-x}\Big), \qquad x\in[0,1).
  \end{equation}
\end{proposition}

Finta \cite{Fi2005} introduced the operator
\begin{equation} \label{eq:2.8}
  \begin{gathered}
    V_n(\f,\x) = \sum _{k=0}^{\infty} \,v_{n,k}(\f)\cP_{n,k}(\x), \qquad \x\in[0,\infty), \\
    v_{n,0}(\f) = \f(0), \qquad v_{n,k}(\f) = (n+1)\int_0^{\infty} \cP_{n+2,k-1}(t)\f(t)\,dt,
  \end{gathered}
\end{equation}
where $\cP_{n,k}$ are the Baskakov basis functions given in \eqref{eq:2.6}. The modification $V_n \f$
of the Baskakov operator is similar to the Goodman-Sharma type variant of the classical Bernstein
operator. Then the operators \eqref{eq:1.3} and \eqref{eq:2.8} satisfy the next property.

\begin{proposition}[{{\cite[Proposition~8]{IvPa2012}}}] \label{pr:2.3}
  For every integrable function $f$ such that one of the series (representing $M_n f$ and $V_n \f$)
  in \eqref{eq:2.9} is convergent and for every $n\in\mathbb{N}$ we have
  \begin{equation} \label{eq:2.9}
    M_n(f,x) % = \sT\big(V_n(\sT^{-1} f,x)\big)
             = (1-x)V_n\Big(\f,\frac{x}{1-x}\Big), \qquad x\in[0,1).
  \end{equation}
\end{proposition}

In \cite{GaPaUl2025} the authors introduced and investigated the approximation properties of the
modified Baskakov-GS operator
\begin{equation} \label{eq:2.10}
  \begin{gathered}
    \wV_n(\f,\x) = \sum_{k=0}^{\infty} v_{n,k}(\f)\wcP_{n,k}(\x), \qquad \x\in[0,\infty), \\
    \wcP_{n,k}(\x) = \cP_{n,k}(\x) - \frac{1}{n}\,\wcD P_{n,k}(\x).
  \end{gathered}
\end{equation}
It was proved in \cite[Theorem~1]{GaPaUl2025} that the modified Baskakov-GS operator $\wV_n$ has
second order rate of approximation for continuous functions, compared with the first order for the
Finta's operator \eqref{eq:2.8}.

A similar relationship between the MKZ operator $M_n^{[MKZ]}$ defined in \eqref{eq:1.1} and the
Baskakov operator $\cB_n$ as in Proposition~\ref{pr:2.2} holds for the modified MKZ-GS operator
$\wM_n$ and the modified Baskakov-GS operator $\wV_n$. The statement is an immediate consequence
from \eqref{eq:1.5} and \eqref{eq:2.9}--\eqref{eq:2.10}.

\begin{proposition} \label{pr:2.4}
  For every integrable function $f$ such that one of the series (representing $\wM_n f$ and
  $\wV_n \f$) in \eqref{eq:2.11} is convergent and for every $n\in\mathbb{N}$ we have
  \begin{equation} \label{eq:2.11}
    \wM_n(f,x) % = \sT\big(\wV_n(\sT^{-1} f,x)\big)
               = (1-x)\wV_n\Big(\f,\frac{x}{1-x}\Big), \qquad x\in[0,1).
  \end{equation}
\end{proposition}

Combining Propositions~\ref{pr:2.1}--\,\ref{pr:2.4} we summarize as follows.
\begin{proposition} \label{pr:2.5}
  Let $\w(\x)=(1+\x)^{-1}$, $\x \in [0,\infty)$. For every integrable function $f$ on $[0,1)$ such
  that one of the series (representing $\wM_n f$ and $\wV_n \f$) in \eqref{eq:2.11} is convergent,
  for the corresponding from \eqref{eq:2.1}--\eqref{eq:2.2} function $\f$ and for every
  $n\in\mathbb{N}$ we have
  \begin{align} \label{eq:2.12}
    \|f\|_{[0,1)} & = \|\w \f\|_{[0,\infty)}, \\
    \big\|\wM_nf\big\|_{[0,1)} & = \big\|\w \wV_n\f\big\|_{[0,\infty)}, \label{eq:2.13} \\
    \big\|\wD \wM_nf\big\|_{[0,1)} & = \big\|\w \wcD \wV_n\f\big\|_{[0,\infty)}. \label{eq:2.14}
  \end{align}
\end{proposition}

%====================================================================================================

\medskip
\section{Auxiliary Results for the modified MKZ-GS operator} \label{s3}

We start this section with some properties of the MKZ-GS operator $M_n$ from \cite{IvPa2012}:

\begin{itemize}
  \item $M_n$ is a linear positive operator;
  \item $M_n(e_0,x)=e_0(x)$, $M_n(e_1,x)=e_1(x)$, where $e_k(t)=t^k$, $k=0,1$;
  \item $M_n f \le f$ for every concave continuous function $f$.
\end{itemize}

\smallskip
Setting $\alpha_0=\alpha_1=0$ in \cite[eq. (4)]{IvPa2012} we get partial cases of
Propositions~9--14 in \cite{IvPa2012} listed below, respectively.

\begin{proposition} \label{pr:3.1}
  \!
  \begin{enumerate}
    \item[(a)] If $\,f\in C[0,1)$ and $n\in\mathbb{N}$, then
          $$
            M_n(f,x) - M_{n+1}(f,x) = \frac{1}{n(n+1)}\,\wD M_n(f,x), \qquad x\in[0,1).
          $$

    \item[(b)] For every $g\in W^2_0(\varphi)[0,1)$ and $n\in\mathbb{N}$ we have
          $$
            \wD M_n(g,x) = M_n(\wD g)(x), \qquad x\in[0,1),
          $$
          i.e. $M_n$ commutes with the operator $\varphi D^2$ on $W^2_0(\varphi)[0,1)$.

    \smallskip
    \item[(c)] Let $m,n\in\mathbb{N}$. Then for every $f\in C[0,1)$ we have $M_n M_m f=M_m M_n f$, i.e.
          $M_m$ and $M_n$ commute on $C[0,1)$.

    \smallskip
    \item[(d)] For every $f\in C[0,1)$ and $n\in\mathbb{N}$ we have $\|M_n f\|\le \|f\|$,
          i.e. $\|M_n\|=1$ holds for the norm of the operator $M_n$ on $C[0,1)$.

    \smallskip
    \item[(e)] For every $g\in W^2(\varphi)[0,1)$ and $n\in\mathbb{N}$ we have
          $\|\wD M_n g\|\le \|\wD g\|$.

    \smallskip
    \item[(f)] {\em Jackson-type inequality:} If $\,g\in W^2(\varphi)[0,1)$ and $n\in \mathbb{N}$ we have
          $$
            \| M_n g - g\| \le \frac{1}{n}\,\|\wD g\|.
          $$
  \end{enumerate}
\end{proposition}

Now, we give properties of the operators $M_n$, $\wM_n$ and the differential operator $\wD$.

\begin{proposition} \label{pr:3.2}
  If the operators $M_n$, $\wM_n$ and the differential operator $\wD$ are defined as in
  \eqref{eq:1.3}, \eqref{eq:1.5} and \eqref{eq:1.4}, respectively, then

  \smallskip
  \begin{enumerate}
    \item[(a)] $\wM_n f = M_n\big(f-\frac{1}{n}\,\wD f\big)$, for $f\in W^2_0(\varphi)[0,1)$;

    \smallskip
    \item[(b)] $\wD \wM_n f = \wM_n \wD f$, for $f\in W^2_0(\varphi)[0,1)$;

    \smallskip
    \item[(c)] $M_n \wM_n f = \wM_n M_n f$, for $f\in W^2_0(\varphi)[0,1)$;

    \smallskip
    \item[(d)] $\wM_m \wM_n f = \wM_n \wM_m f$, for $f\in W^2_0(\varphi)[0,1)$;

    \smallskip
    \item[(e)] $\lim\limits_{n\to \infty} \wM_n f = f$, for $f\in W^2(\varphi)[0,1)$.
  \end{enumerate}
\end{proposition}

\begin{proof}
We have
\begin{align*}
  \wM_{n} f & = \sum_{k=0}^{\infty} u_{n,k}(f)\wP_{n,k}
              = \sum_{k=0}^{\infty} u_{n,k}(f)\Big(P_{n,k} - \frac{1}{n}\,\wD P_{n,k}\Big) \\
            & = \sum_{k=0}^{\infty} u_{n,k}(f) P_{n,k} - \frac{\varphi}{n}\sum_{k=0}^{\infty} u_{n,k}(f)D^2P_{n,k}
              = M_n f - \frac{1}{n}\,\varphi D^2M_n f \\
            & = M_n f - \frac{1}{n}\,\wD M_n f.
\end{align*}
Then from Proposition~\ref{pr:3.1}\,(b) we obtain
$$
  \wM_n f = M_n f - \frac{1}{n}\,\wD M_n f = M_n f - \frac{1}{n}\,M_n \wD f
          = M_n\Big(f-\frac{1}{n}\,\wD f\Big),
$$
which proves (a).

Now, commutative properties (b) and (c) follow from (a) and Proposition~\ref{pr:3.1}\,(b):
\begin{align*}
  \wD \wM_n f  & = \wD M_n\Big(f-\frac{1}{n}\,\wD f\Big) = M_n\Big(\wD f - \frac{1}{n}\,\wD \wD f\Big)
                           = \wM_n(\wD f), \\
  M_n \wM_n f & = M_n\Big(M_nf-\frac{1}{n}\,\wD M_n f \Big) = M_n M_n f - \frac{1}{n}\,M_n \wD M_n f \\
                        & = M_n M_n f - \frac{1}{n}\,\wD M_n M_n f = M_n (M_n f)-\frac{1}{n}\,\wD M_n (M_n f)
                            = \wM_n M_n f.
\end{align*}

The modified MKZ-GS operators commute in the sense of (d) since
\begin{align*}
  \wM_m \wM_n f & = \wM_m \Big(M_n f - \frac{1}{n}\,\wD M_n f\Big) \\
                & = M_m \Big(M_n f - \frac{1}{n}\,\wD M_n f\Big)
                    -\frac{1}{m} \wD M_m\Big(M_n f -\frac{1}{n}\,\wD M_n f\Big) \\
                & = M_m M_n \Big(f - \frac{m+n}{mn}\,\wD f + \frac{1}{mn}\,\wD^2 f\Big).
\end{align*}
The same expression in the last line we obtain for $\wM_n \wM_m f$ because of property (a) and
Proposition~\ref{pr:3.1}\,(b)--(c).

Applying Proposition~\ref{pr:3.1}\,(b),\,(e)--(f) we have
$$
  \|\wM_n f - f\| = \Big\|M_n f  -\frac{1}{n}\,M_n \wD f - f\Big\|
                   \le \|M_n f - f\| + \frac{1}{n}\,\big\|M_n \wD f\big\|
                   \le \frac{2}{n}\,\|\wD f\|.
$$
Therefore $\lim\limits_{n\to \infty} \|\wM_n f - f\|=0$, i.e. property (e) holds true.
\end{proof}

%====================================================================================================

\medskip
\section{Auxiliary Results for the Baskakov Basis Functions $\cP_{n,k}$} \label{s4}

Let $\x\in [0,\infty)$. It is easy to verify that the Baskakov basis functions $\cP_{n,k}$ satisfy
\begin{equation} \label{eq:4.1}
    k \cP_{n,k}(\x) = n\x\,\cP_{n+1,k-1}(\x), \qquad (n+k)\cP_{n,k}(\x) = n(1+\x)\cP_{n+1,k}(\x),
\end{equation}
setting $\cP_{n,k}:=0$ if $k<0$. The central moments of the Baskakov operator $\cB_n$ are
$$
  \mu_{n,j}(\x) = \cB_n\big((t-\x)^j,\x\big)
               = \sum_{k=0}^{\infty} \Big(\frac{k}{n}-\x\Big)^j \cP_{n,k}(\x), \qquad j=0,1,\ldots\,,
$$
with
\begin{equation} \label{eq:4.2}
  \mu_{n,0}(\x) = 1, \quad \mu_{n,1}(\x) = 0, \quad
  \mu_{n,2}(\x) = \frac{\psi(\x)}{n}, \qquad \mu_{n,3}(\x) = \frac{(1+2\x)\psi(\x)}{n^2}.
\end{equation}

Now we introduce a function that relates $\wcD \cP_{n,k}$ and $\cP_{n,k}$, namely
\begin{align} \label{eq:4.3}
  T_{n,k}(\x) & := k(k-1)\frac{1+\x}{\x} - 2k(n+k) + (n+k)(n+k+1)\frac{\x}{1+\x} \\
              & = n\bigg[-1-\frac{1+2\x}{\psi(\x)}\Big(\frac{k}{n}-\x\Big) + \frac{n}{\psi(\x)}\Big(\frac{k}{n}-\x\Big)^{\!2}\bigg].
                \notag
\end{align}
For its derivatives one can see that
\begin{align}
  T_{n,k}'(\x)  & = - \frac{k(k-1)}{\x^2} + \frac{(n+k)(n+k+1)}{(1+\x)^2}, \label{eq:4.4} \\
  T_{n,k}''(\x) & = \frac{2k(k-1)}{\x^3} - \frac{2(n+k)(n+k+1)}{(1+\x)^3}. \label{eq:4.5}
\end{align}

Below are a few more properties of the function $T_{n,k}(\x)$ that will be useful later on.

\begin{proposition}[{\cite[Proposition~2.3]{GaPaUl2025}}] \label{pr:4.1}

  \smallskip
  \! Let $\x\in [0,\infty)$.
  \begin{enumerate}
    \item[(a)] The following relation concerning functions $\cP_{n,k}$, $T_{n,k}$ and differential operator
                  $\wcD$ are valid:
                  \begin{equation} \label{eq:4.6}
                    \wcD \cP_{n,k}(\x) = T_{n,k}(\x) \cP_{n,k}(\x).
                 \end{equation}

    \smallskip
    \item[(b)] If $\,\alpha$ is an arbitrary real number, then
                  $$
                      \Phi_n(\alpha) := \sum_{k=0}^{\infty} \Big(\alpha - \frac{1}{n}\,T_{n,k}(\x)\!\Big)^2 \cP_{n,k}(\x)
                                              = \alpha^2 + 2 + \frac{2}{n}\,.
                  $$
  \end{enumerate}
\end{proposition}

\begin{proposition} \label{pr:4.2}
  For all $\x\in [0,\infty)$,
  $$
    \sum_{k=0}^{\infty} \,T_{n,k}(\x)\Big(\frac{k}{n}-\x\Big)\,\cP_{n,k}(\x) = 0.
  $$
\end{proposition}

\begin{proof}
From \eqref{eq:4.2} and \eqref{eq:4.3} we obtain
$$
  \sum_{k=0}^{\infty} \,T_{n,k}(\x)\Big(\frac{k}{n}-\x\Big)\cP_{n,k}(\x)
  = n\bigg[-\mu_{n,1}(\x)-\frac{1+2\x}{\psi(\x)}\,\mu_{n,2}(\x) + \frac{n}{\psi(\x)}\,\mu_{n,3}(\x)\bigg] = 0.
$$
%\begin{align*}
%  \sum_{k=0}^{\infty} \,T_{n,k}(\x)\Big(\frac{k}{n}-\x\Big)\cP_{n,k}(\x)
%    & = n\bigg[-\mu_{n,1}(\x)-\frac{1+2\x}{\psi(\x)}\,\mu_{n,2}(\x) + \frac{n}{\psi(\x)}\,\mu_{n,3}(\x)\bigg] \\
%    & = n\bigg[-0-\frac{1+2\x}{\psi(\x)}\,\frac{\psi(\x)}{n} + \frac{n}{\psi(\x)}\,\frac{(1+2\x)\psi(\x)}{n^2}\bigg] = 0.
%\end{align*}
\end{proof}

%===================================================================================================

\medskip
\section{A Direct Inequality} \label{s5}

We will first prove the next upper estimate for the norm of the modified MKZ-GS operator $\wM_n$
defined in \eqref{eq:1.5}.

\begin{lemma} \label{le:5.1}
  If $\,n\in \mathbb{N},$  $n\ge 2$ and $f\in C[0,1)$, then
  \begin{equation} \label{eq:5.1}
    \big\|\wM_n f\big\| \le \sqrt{6}\,\|f\|, \qquad\mbox{i.e.} \qquad  \|\wM_n\| \le \sqrt{6}.
  \end{equation}
\end{lemma}

\begin{proof}
In view of Proposition~\ref{pr:2.5}, it is sufficient to prove that if $\w(\x)=(1+\x)^{-1}$,
$n\in\mathbb{N}$, $n\ge 2$, and $\f\in C[0,\infty)$, then
\begin{equation} \label{eq:5.2}
  \| \w \wV_n \f\|_{[0,\infty)} \le \,\sqrt{6}\, \|\w\f\|_{[0,\infty)}.
\end{equation}

Obviously, $|v_{n,k}(\w^{-1}\w\f)| \le \|\w\f\|_{[0,\infty)}|v_{n,k}(\w^{-1})|$. Then for
$\x\in [0,\infty)$ we have
\begin{align*}
  \big|\w(\x)\wV_n(\f,\x)\big|
    & \le \w(\x)\sum_{k=0}^\infty |v_{n,k}(\f)|\,\big|\wcP_{n,k}(\x)\big|
      = \w(\x) \sum_{k=0}^\infty |v_{n,k}(\w^{-1}\w\f)|\,\big|\wcP_{n,k}(\x)\big| \\
    & \le \|\w\f\|_{[0,\infty)}\, \w(\x)\sum_{k=0}^\infty v_{n,k}(\w^{-1})\,\big|\wcP_{n,k}(\x)\big| \\
    & = \|\w\f\|_{[0,\infty)}\, \w(\x)\sum_{k=0}^\infty \frac{n+k}n\,\big|\wcP_{n,k}(\x)\big|,
\end{align*}
i.e.
\begin{equation} \label{eq:5.3}
  \big|\w(\x)\wV_n(\f,\x)\big| \le
  \|\w\f\|_{[0,\infty)}\, \w(\x)\sum_{k=0}^\infty \frac{n+k}n\,\big|\wcP_{n,k}(\x)\big|\,.
\end{equation}
From \eqref{eq:2.10} and \eqref{eq:4.6} we get
$$
  \wcP_{n,k}(\x) = \Big(1 - \frac{1}{n}\,T_{n,k}(\x)\Big) \cP_{n,k}(\x).
$$
Then the Cauchy inequality, \eqref{eq:4.2} and Proposition~\ref{pr:4.1}\,(b) with
$\alpha=1$ yield
\begin{align*}
  \w(\x)\sum_{k=0}^{\infty} \frac{n+k}n\,\big|\wcP_{n,k}(\x)\big|
  & = \w(\x)\sum_{k=0}^{\infty} \Big|1 - \frac{1}{n}\,T_{n,k}(\x)\Big| \frac{n+k}n\,\cP_{n,k}(\x) \\
  & \le \w(\x) \sqrt{\sum_{k=0}^{\infty} \Big(1-\frac{1}{n}\,T_{n,k}(\x)\Big)^{\!2} \cP_{n,k}(\x)}\,
             \sqrt{\sum_{k=0}^{\infty} \Big(\frac{k}{n}+1\Big)^{\!2} \cP_{n,k}(\x)} \\
  & = \w(\x) \sqrt{\Phi_n(1)}\,\sqrt{(1+\x)^2+\frac{\psi(\x)}{n}} \\
  & = \w(\x) \sqrt{3 + \frac{2}{n}}\,\sqrt{(1+\x)^2\Big(1+\frac{\x}{n(1+\x)}\Big)} \\
  & \le \sqrt{3 + \frac{2}{n}}\,\sqrt{1+\frac{1}{n}}\,.
\end{align*}
Hence for $\x\in [0,\infty)$  we have
\begin{equation} \label{eq:5.4}
  \w(\x)\sum_{k=0}^{\infty} \frac{n+k}n\,\big|\wcP_{n,k}(\x)\big| \le \sqrt{6}, \qquad n\ge 2.
\end{equation}

Now, inequality \eqref{eq:5.2} follows immediately from \eqref{eq:5.3} and \eqref{eq:5.4}.

To complete the proof of the lemma, we apply \eqref{eq:5.2}, \eqref{eq:2.12} and \eqref{eq:2.13}:
$$
  \big\|\wM_n f\big\| =   \| \w \wV_n \f\|_{[0,\infty)}
                      \le \sqrt{6}\, \|\w\f\|_{[0,\infty)}
                      = \sqrt{6}\, \|f\|.
$$
\end{proof}

To prove a direct theorem for the approximation rate for functions $f$ by the modified MKZ-GS
operator $\wM_n f$ we need a Jackson type inequality.

\begin{lemma} \label{le:5.2}
  If $\,n\in \mathbb{N}$, $f\in W^2_0(\varphi)[0,1)$ and $\wD f\in W^2(\varphi)[0,1)$, then
  $$
    \big\|\wM_n f - f\big\| \le \frac{1}{n^2}\,\|\wD^2 f\|.
  $$
\end{lemma}

\begin{proof}
Combining Proposition~\ref{pr:3.1}\,(a) with Proposition~\ref{pr:3.2}\,(a)--(b) we obtain
\begin{align*}
  \wM_k f - \wM_{k+1} f
    & = M_k f - \frac{1}{k}\,\wD M_k f - M_{k+1} f + \frac{1}{k+1}\,\wD M_{k+1} f \\
    & = \frac{1}{k(k+1)}\wD M_k f - \frac{1}{k}\,\wD M_k f + \frac{1}{k+1}\,\wD M_{k+1} f \\
%   & = \Big(\frac{1}{k}-\frac{1}{k+1}\Big)\wD M_k f + \frac{1}{k+1}\,\wD M_{k+1} f - \frac{1}{k}\,\wD M_k f \\
    & = - \frac{1}{k+1}\big(\wD M_k f - \wD M_{k+1} f\big)
      = - \frac{1}{k+1}\big(M_k \wD f - M_{k+1} \wD f\big) \\
    & = - \frac{1}{k+1}\cdot\frac{1}{k(k+1)}\, \wD M_k \wD f,
\end{align*}
i.e.
$$
  \wM_k f - \wM_{k+1} f = - \frac{1}{k(k+1)^2}\,\wD M_k \wD f.
$$
Therefore for every $s>n$ we have
\begin{gather*}
 \wM_n f - \wM_s f = \sum_{k=n}^{s-1} \big(\wM_k f - \wM_{k+1} f\big)
                   = - \sum_{k=n}^{s-1} \frac{1}{k(k+1)^2}\,\wD M_k \wD f.
\end{gather*}
Letting $s\to\infty$ by Proposition~\ref{pr:3.2}\,(a) and (e) we obtain
\begin{equation} \label{eq:5.5}
  \wM_n f - f = - \sum_{k=n}^{\infty} \frac{1}{k(k+1)^2}\,\wD M_k \wD f
              = - \sum_{k=n}^{\infty} \frac{1}{k(k+1)^2}\,M_k \wD^2 f.
\end{equation}
Then Proposition~\ref{pr:3.1}\,(e) yields
$$
  \|\wM_n f - f\| \le \sum_{k=n}^{\infty} \frac{1}{k(k+1)^2}\,\big\|\wD M_k \wD f\big\|
                  \le \sum_{k=n}^{\infty} \frac{1}{k(k+1)^2}\,\big\|\wD^2 f\big\|.
$$
Since $\sum_{k=n}^{\infty} \frac{1}{k(k+1)^2}\le \frac{1}{n^2}$
(see e.g. \cite[Proposition~5, eq.~(20)]{GaPaUl2025}) we conclude that
$$
  \big\|\wM_n f - f\big\| \le \frac{1}{n^2}\,\big\|\wD^2 f\big\|.
$$
\end{proof}

%A direct result on the approximation rate of functions $f\in C[0,1)$ by the operators
%\eqref{eq:1.5} in means of the K-functional \eqref{eq:1.6} follows immediately from both
%lemmas above.

\smallskip
\begin{proof}[Proof of Theorem~\ref{th:1.1}]
Let $g$ be any function with $g\in W^2_0(\varphi)[0,1)$ and $\wD g\in W^2(\varphi)[0,1)$.
Then by Lemma~\ref{le:5.1} and Lemma~\ref{le:5.2} we obtain
\begin{align*}
  \big\|\wM_n f-f\big\|
    & \le \big\|\wM_n f-\wM_n g\big\| + \big\|\wM_n g-g\big\| + \|g-f\| \\
    & \le \sqrt{6}\,\|f-g\| + \frac{1}{n^2}\,\big\|\wD^2 g\big\| + \|f-g\| \\
    & \le \big(\sqrt{6}+1\big)\Big(\|f-g\|+\frac{1}{n^2}\,\big\|\wD^2 g\big\|\Big).
\end{align*}
Taking the infimum over all functions $g\in W^2_0(\varphi)[0,1)$ with $\wD g\in W^2(\varphi)[0,1)$
we obtain
$$
  \big\|\wM_n f-f\big\| \le \big(\sqrt{6}+1\big)\,K\Big(f,\frac{1}{n}\Big).
$$
\end{proof}

%====================================================================================================

\bigskip
\section{A Strong Converse Result} \label{s6}

First, we will prove a Voronovskaya type result for the operator $\wM_n$.

\begin{lemma} \label{le:6.1}
  If $\,n\ge 2$, $n\in \mathbb{N}$, $\,\lambda(n) = \sum\limits_{k=n}^{\infty} \dfrac{1}{k(k+1)^2}$,
  $\,\theta(n) = \sum\limits_{k=n}^{\infty} \dfrac{1}{k^2(k+1)^2}$, and $f\in C[0,1)$ is such that
  $f,\,\wD f\in W^2_0(\varphi)[0,1)$ and $\wD^3 f\in L_{\infty}[0,1)$, then
  \begin{equation} \label{eq:6.1}
    \big\|\wM_n f - f + \lambda(n)\wD^2 f\big\| \le \theta(n)\,\big\|\wD^3 f\big\|.
  \end{equation}
\end{lemma}

\begin{proof}
We have from \eqref{eq:5.5} that
$$
  \wM_n f - f + \lambda(n)\wD^2 f
    = - \sum_{k=n}^{\infty} \frac{M_k\wD^2 f}{k(k+1)^2} + \sum_{k=n}^{\infty} \frac{\wD^2 f}{k(k+1)^2}
    = \sum_{k=n}^{\infty} \frac{\wD^2 f - M_k \wD^2 f}{k(k+1)^2}.
$$
Then
$$
  \big\|\wM_n f - f + \lambda(n)\wD^2 f\big\|
  \le \sum_{k=n}^{\infty} \frac{1}{k(k+1)^2}\,\big\|\wD^2 f - M_k \wD^2 f\big\|.
$$
Using the Jackson inequality Proposition~\ref{pr:3.1}\,(f) with $\wD^2 f$ replacing $g$ we obtain
$$
  \big\|\wM_n f - f + \lambda(n)\wD^2 f\big\|
    \le \sum_{k=n}^{\infty} \frac{1}{k(k+1)^2}\cdot\frac{1}{k}\,\big\|\wD \wD^2 f\big\|
    = \theta(n)\,\big\|\wD^3 f\big\|.
$$
\end{proof}

The following is an inequality of Bernstein type.

\begin{lemma} \label{le:6.2}
  If $n\in\mathbb{N}$, $n\ge 17$ and $f\in C[0,1)$, then
  \begin{equation} \label{eq:6.2}
    \|\wD \wM_n f\| \le \wC\, n\|f\|, \qquad \text{where} \ \wC=17.
  \end{equation}
\end{lemma}

\begin{proof}
Because of Proposition~\ref{pr:2.5}, it is sufficient to prove that if $\w(\x)=(1+\x)^{-1}$,
$n\in\mathbb{N}$, and $\f\in C[0,\infty)$, then for $n\ge 17$ the inequality
\begin{equation} \label{eq:6.3}
  \| \w\wcD \wV_n \f\|_{[0,\infty)} \le \wC\, n\|\w\f\|_{[0,\infty)}
\end{equation}
holds true, where $\wC=17$\,.

Similarly to the proof of Lemma~\ref{le:5.1}, for $\x\in [0,\infty)$ we have
%\begin{align*}
%  \big|\w(\x) \wcD \wV_n(\f,\x)\big|
%    & \le \w(\x) \sum_{k=0}^\infty |v_{n,k}(\f)|\,\big|\wcD\wcP_{n,k}(\x)\big| \\
%    & \le \w(\x) \sum_{k=0}^\infty |v_{n,k}(\w^{-1}\w\f)|\,\big|\wcD\wcP_{n,k}(\x)\big| \\
%   & = \|\w\f\|_{[0,\infty)} \w(\x) \sum_{k=0}^\infty v_{n,k}(w^{-1})\,\big|\wcD\wcP_{n,k}(\x)\big| \\
%    & \le \|\w\f\|_{0,\infty)} \w(\x) \sum_{k=0}^\infty \frac{n+k}n\,\big|\wcD\wcP_{n,k}(\x)\big|,
%\end{align*}
$$
  \big|\w(\x) \wcD \wV_n(\f,\x)\big| \le \w(\x) \sum_{k=0}^\infty |v_{n,k}(\f)|\,\big|\wcD\wcP_{n,k}(\x)\big|,
$$
hence
\begin{equation} \label{eq:6.4}
  \big|\w(\x) \wcD\wV_n(\f,\x)\big| \le \|\w\f\|_{[0,\infty)} \w(\x) \sum_{k=0}^\infty \frac{n+k}n\,\big|\wcD\wcP_{n,k}(\x)\big|.
\end{equation}

In order to estimate the expression on the right-hand side of \eqref{eq:6.4} we use the relation
\begin{align*}
  \wcD \wcP_{n,k}(\x) & = \frac{\psi(\x)}{n}\,T_{n,k}''(\x) \cP_{n,k}(\x)
                                        + 2\big[T_{n+1,k-1}(\x)\cP_{n+1,k-1}(\x) +T_{n+1,k}(\x)\cP_{n+1,k}(\x)\big] \\
                      & \qquad + \Big(1 - \frac{1}{n}\,T_{n,k}(\x)\!\Big) T_{n,k}(\x) \cP_{n,k}(\x),
\end{align*}
see the proof of Lemma~4, p.~11, in \cite{GaPaUl2025}. Therefore
$$
   \w(\x) \sum_{k=0}^{\infty} \frac{n+k}n\,\big|\wcD \wcP_{n,k}(\x)\big| \le \ca_n(\x) + \cb_n(\x) + \cc_n(\x),
$$
where
\begin{align*}
  \ca_n(\x) & = \w(\x) \frac{\psi(\x)}{n} \sum_{k=0}^{\infty} \big|T_{n,k}''(\x)\big| \frac{n+k}n\,\cP_{n,k}(\x), \\
  \cb_n(\x) & = 2\w(\x) \sum_{k=0}^{\infty} |T_{n+1,k-1}(\x) | \frac{n+k}n\,\cP_{n+1,k-1}(\x) \\
            & \qquad + 2\w(\x) \sum_{k=0}^{\infty} |T_{n+1,k}(\x)| \frac{n+k}n\,\cP_{n+1,k}(\x), \\
  \cc_n(\x) & = \w(\x) \sum_{k=0}^{\infty} \Big|\Big(1 - \frac{1}{n}\,T_{n,k}(\x)\!\Big) T_{n,k}(\x)\Big| \frac{n+k}n\,\cP_{n,k}(\x).
\end{align*}

\smallskip
{\sl (i) \ Estimation of $\ca_n(\x)$.}
From \cite[Proof of Lemma~4, eq.~(30)]{GaPaUl2025} we have
\begin{equation} \label{eq:6.5}
  \frac{\psi(\x)}{n} \sum_{k=0}^{\infty} |T_{n,k}''(\x)| \cP_{n,k}(\x) \le 6n, \qquad\text{for all} \ \ n\ge 17, \ \ \x\in [0,\infty).
\end{equation}
Since $T_{n,k}''(\x) = T_{n+1,k}''(\x) + \frac{4(n+k+1)}{(1+\x)^3}$, combining with \eqref{eq:4.1}
and \eqref{eq:6.5} we have
\begin{align*}
  \ca_n(\x) & = \frac{\psi(\x)}{n} \sum_{k=0}^{\infty} |T_{n,k}''(\x)| \cP_{n+1,k}(\x) \\
            & = \frac{\psi(\x)}{n} \sum_{k=0}^{\infty} \bigg|T_{n+1,k}''(\x)+\frac{4(n+k+1)}{(1+\x)^3}\bigg| \cP_{n+1,k}(\x) \\
            & \le \frac{n+1}{n}\, \frac{\psi(\x)}{n+1} \sum_{k=0}^{\infty} \big|T_{n+1,k}''(\x)\big| \cP_{n+1,k}(\x)+ \frac{4\psi(\x)}{n} \sum_{k=0}^{\infty} \frac{n+k+1}{(1+\x)^3}\, \cP_{n+1,k}(\x) \\
            & \le \frac{n+1}{n}\,6(n+1)+\frac{4\psi(\x)}{(1+\x)^3}\,\frac{n+1}{n} \sum_{k=0}^{\infty} \Big(1+\frac{k}{n+1}\Big) \cP_{n+1,k}(\x) \\
            & = \frac{6(n+1)^2}{n} + \frac{4\psi(\x)}{(1+\x)^3}\,\frac{n+1}{n}\,(1+\x) \\
            & = \frac{6(n+1)^2}{n} + \frac{4\x}{(1+\x)}\,\frac{n+1}{n} \\
            & \le \frac{6(n+1)^2}{n}+\frac{4(n+1)}{n}\,.
\end{align*}
Hence for $n\ge 14$,
\begin{equation} \label{eq:6.6}
  \ca_n(x) \le 7\,n, \qquad x\in [0,\infty).
\end{equation}

\smallskip
{\sl (ii) \ Estimation of $\cb_n(\x)$.} We have

\begin{align*}
  \frac{\cb_n(\x)}{2\w(\x)} & = \sum_{k=0}^{\infty} |T_{n+1,k-1}(\x) | \frac{n+k}n\,\cP_{n+1,k-1}(\x) + \sum_{k=0}^{\infty} |T_{n+1,k}(\x)| \frac{n+k}n\,\cP_{n+1,k}(\x) \\
                 & = \sum_{k=0}^{\infty} |T_{n+1,k}(\x) | \frac{n+1+k}n\,\cP_{n+1,k}(\x) + \sum_{k=0}^{\infty} |T_{n+1,k}(\x)| \frac{n+k}n\,\cP_{n+1,k}(\x) \\
                 & = \frac{n+1}{n} \sum_{k=0}^{\infty} |T_{n+1,k}(\x)| \frac{n+1+k}{n+1}\,\cP_{n+1,k}(\x) \\
                 & \qquad + \frac{n+1}{n} \sum_{k=0}^{\infty} |T_{n+1,k}(\x)| \Big(\frac{n+1+k}{n+1}-\frac{1}{n+1}\Big)\,\cP_{n+1,k}(\x) \\
                 & = \frac{2(n+1)}{n} \sum_{k=0}^{\infty} |T_{n+1,k}(\x)| \frac{n+1+k}{n+1}\,\cP_{n+1,k}(\x)
                     - \frac{1}{n} \sum_{k=0}^{\infty} |T_{n+1,k}(\x)|\,\cP_{n+1,k}(\x) \\
                 & \le \frac{2(n+1)}{n} \sum_{k=0}^{\infty} |T_{n+1,k}(\x)| \frac{n+1+k}{n+1}\,\cP_{n+1,k}(\x)
\end{align*}

By the Cauchy inequality and Proposition~\ref{pr:4.1}\,(b) with $\alpha=0$ we obtain
\begin{align*}
  \cb_n(\x) & \le 4\w(\x) \frac{n+1}{n} \sqrt{\sum_{k=0}^{\infty} T_{n+1,k}^2(\x) \cP_{n+1,k}(\x)}\,\sqrt{\sum_{k=0}^{\infty} \Big(\frac{n+1+k}{n+1}\Big)^2 \cP_{n+1,k}(\x)} \\
            & = 4\w(\x)\,\frac{n+1}{n}\sqrt{(n+1)^2\,\Phi_{n+1}(0)}\:\sqrt{(1+\x)^2+\frac{\psi(\x)}{n+1}} \\
            & = \frac{4(n+1)^2}{n}\,\sqrt{2+\frac2{n+1}}\,\sqrt{1+\frac{\x}{(1+\x)(n+1)}} \\
            & \le \frac{4\sqrt{2}\,(n+1)(n+2)}{n}\,.
\end{align*}
Obviously, for all $n\ge 17$,
\begin{equation} \label{eq:6.7}
  \cb_n(\x) \le 7\,n, \qquad \x\in [0,\infty).
\end{equation}

\smallskip
{\sl (iii) \ Estimation of $\cc_n(\x)$.} Let us first point out that from \eqref{eq:4.3} it follows
\begin{equation} \label{eq:6.8}
  T_{n,k}(\x) = T_{n+1,k}(\x) + \frac{2(n+1)}{1+\x}\Big(\frac{k}{n+1}-\x\Big).
\end{equation}

By applying \eqref{eq:4.1} and the Cauchy inequality we have
\begin{align*}
  \cc_n(\x) % & = \w(\x) \sum_{k=0}^{\infty} \Big|\Big(1 - \frac{1}{n}\,T_{n,k}(\x)\!\Big) T_{n,k}(\x)\Big| \frac{n+k}n\,\cP_{n,k}(\x) \\
                & = \sum_{k=0}^{\infty} \Big|\Big(1 - \frac{1}{n}\,T_{n,k}(\x)\!\Big) T_{n,k}(\x)\Big| \frac{n+k}{n(1+\x)}\,\cP_{n,k}(\x) \\
                & = \sum_{k=0}^{\infty} \Big|T_{n,k}(\x)\Big(1 - \frac{1}{n}\,T_{n,k}(\x)\Big)\Big| \cP_{n+1,k}(\x) \\
                & \le \sqrt{\sum_{k=0}^{\infty} T_{n,k}^2(\x) \cP_{n+1,k}(\x)}\,
                        \sqrt{\sum_{k=0}^{\infty} \Big(1-\frac{1}{n}\,T_{n,k}(\x)\Big)^{\!2} \cP_{n+1,k}(\x)}.
\end{align*}

From \eqref{eq:6.8} we obtain
\begin{align*}
  & \sum_{k=0}^{\infty} T_{n,k}^2(\x) \cP_{n+1,k}(\x) \\
  & \qquad = \sum_{k=0}^{\infty} \Big[T_{n+1,k}(\x)+\frac{2(n+1)}{1+\x}\Big(\frac{k}{n+1}-\x\Big)\Big]^2\,\cP_{n+1,k}(\x) \\
  & \qquad = \sum_{k=0}^{\infty} T_{n+1,k}^2(\x) \cP_{n+1,k}(\x) + \frac{4(n+1)}{1+\x} \sum_{k=0}^{\infty} T_{n+1,k}(\x) \Big(\frac{k}{n+1}-\x\Big) \cP_{n+1,k}(\x) \\
  & \qquad\qquad + \frac{4(n+1)^2}{(1+\x)^2} \sum_{k=0}^{\infty} \Big(\frac{k}{n+1}-\x\Big)^{\!2} \cP_{n+1,k}(\x).
\end{align*}

Then \eqref{eq:4.2}, Proposition~\ref{pr:4.2} and Proposition~\ref{pr:4.1}\,(b) with $\alpha=0$ yield
\begin{align*}
  \sum_{k=0}^{\infty} T_{n,k}^2(\x) \cP_{n+1,k}(\x)
    & \le (n+1)^2\,\Phi_{n+1}(0) + 0 + \frac{4(n+1)^2}{(1+\x)^2}\,\mu_{n+1,2}(\x) \\
    & = (n+1)^2\Big(2+\frac{2}{n+1}\Big) + \frac{4(n+1)^2}{(1+\x)^2}\cdot\frac{\psi(\x)}{n+1} \\
    & \le 2(n+1)(n+2) + 4(n+1) \\
    & = 2(n+1)(n+4)\,.
\end{align*}
Moreover, using \eqref{eq:6.8} we get
\begin{align*}
  & \sum_{k=0}^{\infty} \Big(1-\frac{1}{n}\, T_{n,k}(\x) \Big)^{\!2} P_{n+1,k}(\x) \\
  & \qquad = \sum_{k=0}^{\infty} \Big[1-\frac{1}{n}\, T_{n+1,k}(\x) -\frac{2(n+1)}{n(1+\x)}\Big(\frac{k}{n+1}-\x\Big)\Big]^2\,\cP_{n+1,k}(\x) \\
  & \qquad = \sum_{k=0}^{\infty} \Big(1-\frac{1}{n}\, T_{n+1,k}(\x) \Big)^{\!2}\cP_{n+1,k}(\x)
                     - \frac{4(n+1)}{n(1+\x)} \sum_{k=0}^{\infty} \Big(\frac{k}{n+1}-\x\Big) \cP_{n+1,k}(\x) \\
  & \qquad\quad + \frac{4(n+1)}{n^2(1+\x)} \sum_{k=0}^{\infty} T_{n+1,k}(\x)\Big(\frac{k}{n+1}-\x\Big) \cP_{n+1,k}(\x) \\
  & \qquad\quad + \frac{4(n+1)^2}{n^2(1+\x)^2}\sum_{k=0}^{\infty} \Big(\frac{k}{n+1}-\x\Big)^{\!2} \cP_{n+1,k}(\x).
\end{align*}

Now, \eqref{eq:4.2}, Proposition~\ref{pr:4.2} and Proposition~\ref{pr:4.1}\,(b) with $\alpha=1$ yield
\begin{align*}
  \sum_{k=0}^{\infty} \Big(1-\frac1n T_{n,k}(\x) \Big)^{\!2} P_{n+1,k}(\x)
    & = \Phi_{n+1}(1) + 0 + 0 + \frac{4(n+1)^2}{n^2(1+\x)^2}\,\mu_{n+1,2}(\x) \\
    & = 3+\frac{2}{n+1} + \frac{4(n+1)^2}{n^2(1+\x)^2}\cdot\frac{\psi(\x)}{n+1} \\
    & \le 3+\frac{2}{n+1} + \frac{4(n+1)}{n^2}\,.
\end{align*}

Then for all $n\ge 16$ and $\x\in [0,\infty)$,
\begin{equation} \label{eq:6.9}
  \cc_n(x) \le \sqrt{2(n+1)(n+4)}\, \sqrt{3+\frac{2}{n+1} + \frac{4(n+1)}{n^2}}\le 3n.
\end{equation}

From \eqref{eq:6.6}, \eqref{eq:6.7} and \eqref{eq:6.9} we obtain that for all $n\ge 17$,
$$
  \w(\x) \sum_{k=0}^{\infty} \frac{n+k}n\,\big|\wcD \wcP_{n,k}(\x)\big|  \le \ca_n(\x) + \cb_n(\x) + \cc_n(\x) \le 17n, \qquad  \x\in [0,\infty),
$$
i.e.
$$
  \big\|\w\wcD\wV_n \f\big\| \le \wC n\|\w\f\|, \qquad \text{where} \ \wC = 17\,.
$$
\end{proof}

\smallskip
Finally, following the approach of Ditzian and Ivanov~\cite{DiIv1993} we prove a strong converse
inequality.

\smallskip
\begin{proof}[Proof of Theorem~\ref{th:1.2}]

Let $n\in\mathbb{N}$, $n\ge 17$, $f\in C[0,1)$ and $\lambda(n)$, $\theta(n)$ be defined as in
Lemma~\ref{le:6.1}. Applying the Voronovskaya type inequality in Lemma~\ref{le:6.1} for the
operator $\wM_{\ell}$ and function $\wM_n^3 f$ (triple iterated $\wM_n$) replacing $f$ we obtain
\begin{align*}
  \lambda(\ell) \big\|\wD^2 \wM_n^3 f\big\|
    & = \big\|\lambda(\ell)\wD^2 \wM_n^3 f\big\| \\
    & = \big\|\wM_{\ell} \wM_n^3 f - \wM_n^3 f +\lambda(\ell)\wD^2 \wM_n^3 f
        - \wM_{\ell} \wM_n^3 f + \wM_n^3 f\big\| \\
    & \le \big\|\wM_{\ell} \wM_n^3 f - \wM_n^3 f +\lambda(\ell)\wD^2 \wM_n^3 f\big\|
          + \big\|\wM_{\ell} \wM_n^3 f - \wM_n^3 f\big\| \\
    & \le \theta(\ell) \big\|\wD^3 \wM_n^3 f\big\| + \big\|\wM_n^3\big(\wM_{\ell} f - f\big)\big\|.
\end{align*}
Using Lemma~\ref{le:6.2} for the function $\wD^2 \wM_n^2 f$ and successively three times
Lemma~\ref{le:5.1} we get
\begin{align*}
  \lambda(\ell) \big\|\wD^2 \wM_n^3 f\big\|
    & \le \wC\,n\,\theta(\ell) \big\|\wD^2 \wM_n^2 f\big\| + 6\sqrt{6}\,\big\|\wM_{\ell}f - f\big\| \\
    & \le \wC\,n\,\theta(\ell) \big\|\wD^2 \wM_n^2 (f - \wM_n f) + \wD^2 \wM_n^3 f\big\|
          + 15\big\|\wM_{\ell}f - f\big\| \\
    & \le \wC\,n\,\theta(\ell) \big\|\wD^2 \wM_n^2 (f - \wM_n f)\big\|
          + \wC\,n\,\theta(\ell) \big\|\wD^2 \wM_n^3 f\big\| + 15\big\|\wM_{\ell}f - f\big\|.
\end{align*}
Application of the Bernstein type inequality Lemma~\ref{le:6.2} twice for $f-\wM_n f$ yields
$$
  \lambda(\ell) \big\|\wD^2 \wM_n^3 f\big\|
  \le \wC^3 n^3\theta(\ell) \big\|f - \wM_n f\big\|
      + 15\big\|\wM_{\ell} f - f\big\| + \wC\,n\,\theta(\ell) \big\|\wD^2 \wM_n^3 f\big\|.
$$
For $\lambda(n)$ and $\theta(n)$ (see \cite{GaPaUl2025}[Proposition~5]) we have the estimates
$$
  \frac{1}{3n^2} \le \lambda(n) \le \frac{1}{n^2}\,, \qquad
  \theta(n) \le \frac{4}{9n^3}\, \qquad n\ge 2.
$$
Then, applying the lower bound for $\lambda(n)$ we obtain
%$$
%  \frac{1}{3\ell^2} \big\|\wD^2 \wM_n^3 f\big\|
%  \le \frac{4\wC^3 n^3}{9\ell^3} \big\|f - \wM_n f\big\|
%      + 15\big\|\wM_{\ell} f - f\big\| + \frac{4\wC n}{9\ell^3} \big\|\wD^2 \wM_n^3 f\big\|,
%$$
%i.e.
$$
  \frac{1}{\ell^2} \big\|\wD^2 \wM_n^3 f\big\|
  \le \frac{4\wC^3 n^3}{3\ell^3} \big\|f - \wM_n f\big\|
      + 45\big\|\wM_{\ell} - f\big\| + \frac{4\wC n}{3\ell^3} \big\|\wD^2 \wM_n^3 f\big\|.
$$

Now we choose $\ell$ sufficiently large such that
$$
  \frac{4\wC n}{3\ell^3} \le \frac{1}{2\ell^2}, \qquad\text{i.e.}\qquad \ell \ge \frac{8\wC}{3}\,n.
$$
If we set $L=\frac{8\wC}{3}$, for all integers $\ell\ge Ln$ we have
\begin{align}
  \frac{1}{\ell^2} \big\|\wD^2 \wM_n^3 f\big\|
    & \le \frac{4\wC^3 n^3}{3\ell^3} \big\|f - \wM_n f\big\|
        + 45\big\|\wM_{\ell} f - f\big\| + \frac{1}{2\ell^2} \big\|\wD^2 \wM_n^3 f\big\|, \notag\\
  \frac{1}{2\ell^2} \big\|\wD^2 \wM_n^3 f\big\|
    & \le \frac{4\wC^3 n^3}{3\ell^3} \big\|f - \wM_n f\big\| + 45\big\|\wM_{\ell} f - f\big\|, \notag\\
  \frac{1}{n^2} \big\|\wD^2 \wM_n^3 f\big\|
    & \le \wC^2\,\big\|f - \wM_n f\big\|
        + 90\frac{\ell^2}{n^2}\,\big\|\wM_{\ell} f - f\big\|. \label{eq:6.10}
\end{align}
By Lemma~\ref{le:5.1},
\begin{equation} \label{eq:6.11}
  \begin{aligned}
    \big\|f-\wM_n^3 f\big\|
      & \le \big\|f - \wM_n f\big\| + \big\|\wM_n f - \wM_n^2 f\big\| + \big\|\wM_n^2 f - \wM_n^3 f\big\| \\
      & \le \big\|f - \wM_n f\big\| + \sqrt{6}\big\|f - \wM_n f\big\| + 6\big\|f - \wM_n f\big\| \\
      & < 10\big\|f - \wM_n f\big\|.
  \end{aligned}
\end{equation}
Since $\wM_n^3 f\in W^2_0(\psi)$, from \eqref{eq:6.10} and \eqref{eq:6.11} it follows
\begin{align*}
  K\Big(f,\frac{1}{n^2}\Big)
    & = \inf \Big\{\|f-g\| + \frac{1}{n^2}\,\big\|\wD^2 g\big\|: \ g\in W^2_0(\psi),\,\wD g\in W^2(\psi)  \Big\} \\
    & \le \big\|f-\wM_n^3 f\| + \frac{1}{n^2}\,\big\|\wD^2 \wM_n^3 f\big\| \\
    & \le \big(10+\wC^2\big)\big\|\wM_n f-f\big\|
          + 90\,\frac{\ell^2}{n^2}\,\big\|\wM_{\ell} f-f\big\|.
\end{align*}

Hence, we obtain the following upper estimate of the K-functional,
$$
  K\Big(f,\frac{1}{n^2}\Big) \le
  C\,\frac{\ell^2}{n^2}\big(\big\|\wM_n f - f \big\| + \big\|\wM_{\ell} f - f) \big\| \big),
$$
for all $\ell\ge Ln$, where $C=10+\wC^2$ and $L=\frac{8\wC}{3}$, $\wC=17$\,.
\end{proof}

%====================================================================================================

\end{document}